\documentclass[a4paper,12pt]{amsart}
\calclayout
\usepackage{amssymb}
\usepackage{amsmath}
\usepackage{amsfonts}
\usepackage[all]{xy}
\usepackage{mathabx}

\newtheorem{thm}{Theorem}[section]
\newtheorem{cor}[thm]{Corollary}
\newtheorem{lem}[thm]{Lemma}
\newtheorem{prop}[thm]{Proposition}
\theoremstyle{definition}
\newtheorem{defn}[thm]{Definition}
\theoremstyle{remark}

\numberwithin{equation}{section}

\newcommand{\bb}[1]{\mathbb{#1}}
\newcommand{\cl}[1]{\mathcal{#1}}

\begin{document}

\title{On maximal tensor products and quotient maps of operator systems}

\author{Kyung Hoon Han}

\address{Department of Mathematical Sciences, Seoul National University, San 56-1 ShinRimDong, KwanAk-Gu, Seoul
151-747, Korea}

\email{kyunghoon.han@gmail.com}



\date{}

\dedicatory{}

\commby{}


\begin{abstract}
We introduce quotient maps in the category of operator systems and
show that the maximal tensor product is projective with respect to
them. Whereas, the maximal tensor product is not injective, which
makes the $({\rm el},\max)$-nuclearity distinguish a class in the
category of operator systems. We generalize Lance's
characterization of $C^*$-algebras with the WEP by showing that
$({\rm el},\max)$-nuclearity is equivalent to the weak expectation
property. Applying Werner's unitization to the dual spaces of
operator systems, we consider a class of completely positive maps
associated with the maximal tensor product and establish the
duality between quotient maps and complete order embeddings.

\end{abstract}

\maketitle

\section{Introduction}

Kadison characterized the unital subspaces of a real continuous
function algebra on a compact set \cite{Ka}. As for its
noncommutative counterpart, Choi and Effros gave an abstract
characterization of the unital involutive subspaces of $B(H)$
\cite{CE}. The former is called a real function system or a real
ordered vector space with an Archimedean order unit while the
latter is termed an operator system. Ever since the work by Choi
and Effros, the notion of operator systems has been a useful tool
in studying the local structures and the functorial aspects of
$C^*$-algebras.

Recently, the fundamental and systematic developments in the
theory of operator systems have been carried out through a series
of papers \cite{PT, PTT, KPTT1, KPTT2}. Perhaps, the reference
\cite{PT} is the cornerstone for this program. Under the naive
definition of the quotient and that of the tensor product of real
function systems, the order unit sometimes fails to be
Archimedean. See for example, \cite[p.67]{A} and \cite[Remark
3.12]{PTT}. The Archimedeanization process introduced in \cite{PT}
helps remedy this problem. In order to apply this idea to the
noncommutative situation, the Archimedeanization of a matrix
ordered $*$-vector space with a matrix order unit is introduced in
\cite{PTT}. Based on the matricial Archimedeanization process, the
tensor products and the quotients of operator systems are defined
and studied in \cite{KPTT1} and \cite{KPTT2} respectively.

Based on these developments, a simple and generalized proof of the
celebrated Choi-Effros-Kirchberg approximation theorem for nuclear
$C^*$-algebras has been given in \cite{HP}.

In this paper, we continue to study the tensor products in
\cite{KPTT1} and the quotients in \cite{KPTT2}. With the same
spirit as in \cite{HP}, we give a simple and generalized proof of
the classical Lance's theorem \cite{L1,L2}.

In section 3, we introduce the notion of {\sl complete order
quotient maps} which can be regarded as quotient maps in the
category of operator systems and show that the maximal tensor
product is projective under this definition.

The minimal tensor product is injective functorially, while the
maximal tensor product need not be injective. This misbehavior
distinguishes a class which is called $({\rm el},\max)$-nuclear
operator systems \cite{KPTT2}. In the category of $C^*$-algebras,
Lance characterized this tensorial property by the factorization
property for the inclusion map into the second dual through
$B(H)$. In \cite{KPTT2}, it is proved that this factorization
property for an operator system implies its $({\rm
el},\max)$-nuclearity and it is asked whether the converse holds.
In section 4, we answer this question in the affirmative and
deduce Lance's theorem as a corollary.

The order unit of the dual spaces of operator systems cannot be
considered in general. However, the dual spaces of finite
dimensional operator systems have a non-canonical Archimedean
order unit \cite{CE}. This enables the duality between tensor
products and mapping spaces to work in the proofs of the
Choi-Effros-Kirchberg theorem for operator systems \cite{HP} and
Lance's theorem for operator systems in section 4. Not only is the
finite dimensional assumption restrictive, but also the matrix
order unit norm on the dual spaces of finite dimensional operator
systems is irrelevant to the matrix norm given by the standard
dual of operator spaces.

To get rid of the finite dimensional assumption and to reflect the
operator space dual norm, we apply Werner's unitization of matrix
ordered operator spaces~\cite{W} to the dual spaces of operator
systems. We consider the completely positive maps associated with
the maximal tensor products and prove their factorization property
in section 5. Finally, we establish the duality between complete
order quotient maps and complete order embeddings in section 6.

\section{preliminaries}

Let $\cl S$ and $\cl T$ be operator systems. As in \cite{KPTT1},
an operator system structure on $\cl S \otimes \cl T$ is defined
as a family of cones $M_n (\cl S \otimes_\tau \cl T)^+$
satisfying:
\begin{enumerate}
\item[(T1)] $(\cl S \otimes \cl T, \{ M_n (\cl S \otimes_\tau \cl
T)^+ \}_{n=1}^\infty, 1_{\cl S} \otimes1_{\cl T})$ is an operator
system denoted by $\cl S \otimes_\tau \cl T$, \item[(T2)] $M_n(\cl
S)^+ \otimes M_m(\cl T)^+ \subset M_{mn} (\cl S \otimes_\tau \cl
T)^+$ for all $n,m \in \mathbb N$, and \item[(T3)] If $\varphi :
\cl S \to M_n$ and $\psi : \cl T \to M_m$ are unital completely
positive maps, then $\varphi \otimes \psi : \cl S \otimes_\tau \cl
T \to M_{mn}$ is a unital completely positive map.
\end{enumerate}
By an operator system tensor product, we mean a mapping $\tau :
\cl O \times \cl O \to \cl O$, such that for every pair of
operator systems $\cl S$ and $\cl T$, $\tau (\cl S, \cl T)$ is an
operator system structure on $\cl S \otimes \cl T$, denoted $\cl S
\otimes_\tau \cl T$. We call an operator system tensor product
$\tau$ functorial, if the following property is satisfied:
\begin{enumerate}
\item[(T4)] For any operator systems $\cl S_1, \cl S_2, \cl T_1,
\cl T_2$ and unital completely positive maps $\varphi : \cl S_1
\to \cl T_1, \psi : \cl S_2 \to \cl T_2$, the map $\varphi \otimes
\psi : \cl S_1 \otimes \cl S_2 \to \cl T_1 \otimes \cl T_2$ is
unital completely positive.
\end{enumerate}
An operator system structure is defined on two fixed operator
systems, while the functorial operator system tensor product can
be thought of as the bifunctor on the category consisting of
operator systems and unital completely positive maps.

For operator systems $\cl S$ and $\cl T$, we put
$$M_n(\cl S \otimes_{\min} \cl T)^+ = \{ [p_{i,j}]_{i,j} \in
M_n(\cl S \otimes \cl T) : \forall \varphi \in S_k(\cl S), \psi
\in S_m(\cl T), [(\varphi \otimes \psi)(p_{i,j})]_{i,j} \in
M_{nkm}^+ \}$$ and let $\iota_{\cl S} : \cl S \to B(H)$ and
$\iota_{\cl T} : \cl T \to B(K)$ be unital completely order
isomorphic embeddings. Then the family $\{ M_n(\cl S
\otimes_{\min} \cl T)^+ \}_{n=1}^\infty$ is an operator system
structure on $\cl S \otimes \cl T$ rising from the embedding
$\iota_{\cl S} \otimes \iota_{\cl T} : \cl S \otimes \cl T \to B(H
\otimes K)$. We call the operator system $(\cl S \otimes \cl T, \{
M_n(\cl S \otimes_{\min} \cl T) \}_{n=1}^\infty, 1_{\cl S} \otimes
1_{\cl T})$ the minimal tensor product of $\cl S$ and $\cl T$ and
denote it by $\cl S \otimes_{\min} \cl T$.

The mapping $\min : \cl O \times \cl O \to \cl O$ sending $(\cl S,
\cl T)$ to $\cl S \otimes_{\min} \cl T$ is an injective,
associative, symmetric and functorial operator system tensor
product. The positive cone of the minimal tensor product is the
largest among all possible positive cones of operator system
tensor products \cite[Theorem 4.6]{KPTT1}. For $C^*$-algebras $\cl
A$ and $\cl B$, we have the completely order isomorphic inclusion
$$\cl A \otimes_{\min} \cl B \subset \cl A
\otimes_{\rm C^*\min} \cl B$$ \cite[Corollary 4.10]{KPTT1}.

For operator systems $\cl S$ and $\cl T$, we put
$$D_n^{\max}(\cl S, \cl T) = \{ \alpha(P \otimes Q)
\alpha^* : P \in M_k(\cl S)^+, Q \in M_l(\cl T)^+, \alpha \in
M_{n,kl},\ k,l \in \mathbb N \}.$$ Then it is a matrix ordering on
$\cl S \otimes \cl T$ with order unit $1_{\cl S} \otimes 1_{\cl
T}$. Let $\{ M_n(\cl S \otimes_{\max} \cl T)^+ \}_{n=1}^\infty$ be
the Archimedeanization of the matrix ordering $\{ D_n^{\max}(\cl
S, \cl T) \}_{n=1}^\infty$. Then it can be written as
$$M_n(\cl S \otimes_{\max} \cl T)^+ = \{ X \in M_n(\cl S
\otimes \cl T) : \forall \varepsilon>0, X+\varepsilon I_n \in
D_n^{\max}(\cl S, \cl T) \}.$$ We call the operator system $(\cl S
\otimes \cl T, \{ M_n(\cl S \otimes_{\max} \cl T)^+
\}_{n=1}^\infty, 1_{\cl S} \otimes 1_{\cl T})$ the maximal
operator system tensor product of $\cl S$ and $\cl T$ and denote
it by $\cl S \otimes_{\max} \cl T$.

The mapping $\max : \cl O \times \cl O \to \cl O$ sending $(\cl S,
\cl T)$ to $\cl S \otimes_{\max} \cl T$ is an associative,
symmetric and functorial operator system tensor product. The
positive cone of the maximal tensor product is the smallest among
all possible positive cones of operator system tensor products
\cite[Theorem 5.5]{KPTT1}. For $C^*$-algebras $\cl A$ and $\cl B$,
we have the completely order isomorphic inclusion
$$\cl A \otimes_{\max} \cl B \subset \cl A
\otimes_{\rm C^*\max} \cl B$$ \cite[Theorem 5.12]{KPTT1}.

The inclusion $\cl S \otimes \cl T \subset I(\cl S) \otimes_{\max}
\cl T$ for the injective envelope $I(\cl S)$ of $\cl S$ induces
the operator system structure on $\cl S \otimes \cl T$, which is
denoted by $\cl S \otimes_{\rm el} \cl T$. Here the injective
envelope $I(\cl S)$ can be replaced by any injective operator
system containing $\cl S$. The mapping ${\rm el} : \cl O \times
\cl O \to \cl O$ sending $(\cl S, \cl T)$ to $\cl S \otimes_{\rm
el} \cl T$ is a left injective functorial operator system tensor
product \cite[Theorems 7.3, 7.5]{KPTT1}.

An operator system $\cl S$ is called $({\rm el},\max)$-nuclear if
$\cl S \otimes_{\rm el} \cl T = \cl S \otimes_{\max} \cl T$ for
any operator system $\cl T$. An operator system $\cl S$ is $({\rm
el},\max)$-nuclear if and only if $\cl S \otimes_{\max} \cl T
\subset \cl S_2 \otimes_{\max} \cl T$ for any inclusion $\cl S
\subset \cl S_2$ and any operator system $\cl T$ \cite[Lemma
6.1]{KPTT2}. We say that the operator system $\cl S$ has the weak
expectation property (in short, WEP) if the inclusion map $\iota :
\cl S \hookrightarrow \cl S^{**}$ can be factorized as $\Psi \circ
\Phi = \iota$ for unital completely positive maps $\Phi : \cl S
\to B(H)$ and $\Psi : B(H) \to \cl S^{**}$. If $\cl S$ has the
WEP, then it is $({\rm el},\max)$-nuclear \cite[Theorem
6.7]{KPTT2}.

Given an operator system $\cl S$, we call $\cl J \subset \cl S$ a
kernel, provided that it is the kernel of a unital completely
positive map from $\cl S$ to another operator system. The kernel
can be characterized in an intrinsic way: $\cl J$ is a kernel if
and only if it is the intersection of a closed two-sided ideal in
$C^*_u(\cl S)$ with $\cl S$ \cite[Corollary 3.8]{KPTT2}. If we
define a family of positive cones $M_n(\cl S \slash \cl J)^+$ on
$M_n(\cl S \slash \cl J)$ by
$$M_n(\cl S \slash \cl J)^+ = \{ [x_{i,j}+J]_{i,j} :
\forall \varepsilon >0, \exists k_{i,j} \in J,\ \varepsilon I_n
\otimes 1_{\cl S} + [x_{i,j}+k_{i,j}]_{i,j} \in M_n(\cl S)^+ \},$$
then $(\cl S \slash \cl J, \{ M_n(\cl S \slash \cl J)^+
\}_{n=1}^\infty, 1_{\cl S \slash \cl J}) $ satisfies all the
conditions of an operator system \cite[Proposition 3.4]{KPTT2}. We
call it the quotient operator system. With this definition, the
first isomorphism theorem is proved: If $\varphi : \cl S \to \cl
T$ is a unital completely positive map with $\cl J \subset \ker
\varphi$, then the map $\widetilde \varphi : \cl S \slash \cl J
\to \cl T$ given by $\widetilde \varphi (x+\cl J)=\varphi(x)$ is a
unital completely positive map \cite[Proposition 3.6]{KPTT2}.

Since the kernel $\cl J$ in an operator system $\cl S$ is a closed
subspace, the operator space structure of $\cl S \slash \cl J$ can
be interpreted in two ways, one as the operator space quotient and
the other as the operator space structure induced by the operator
system quotient. The two matrix norms can be different. For a
concrete example, see \cite[Example 4.4]{KPTT2}.

\section{Projectivity of maximal tensor product}

We show that the maximal tensor product is projective functorially
in the category of operator systems. To this end, we first define
the quotient maps in the category of operator systems.

\begin{defn}
For operator systems $\cl S$ and $\cl T$, we let $\Phi : \cl S \to
\cl T$ be a unital completely positive surjection. We call $\Phi :
\cl S \to \cl T$ {\sl a complete order quotient map} if for any
$Q$ in $M_n(\cl T)^+$ and $\varepsilon>0$, we can take an element
$P$ in $M_n(\cl S)$ so that it satisfies
$$P + \varepsilon I_n \otimes 1_{\cl S} \in M_n(\cl S)^+
\quad \text{and} \quad \Phi_n(P) = Q.$$
\end{defn}

The key point of the above definition is that the lifting $P$
depends on the choice of $\varepsilon>0$. A slight modification of
\cite[Theorem 2.45]{PT} implies the following proposition that
justifies the above terminology.

\begin{prop}\label{first}
For operator systems $\cl S$ and $\cl T$, we suppose that $\Phi :
\cl S \to \cl T$ is a unital completely positive surjection. Then
$\Phi : \cl S \to \cl T$ is a complete order quotient map if and
only if the induced map $\widetilde{\Phi} : \cl S \slash \ker \Phi
\to \cl T$ is a unital complete order isomorphism.
\end{prop}

\begin{proof}
$\Phi : \cl S \to \cl T$ is a complete order quotient map
\begin{enumerate} \item[$\Leftrightarrow$] $\forall Q \in M_n(\cl T)^+,
\forall \varepsilon>0, \exists P \in M_n(\cl S), P + \varepsilon
I_n \otimes 1_{\cl S} \in M_n(\cl S)^+$ and $\Phi_n(P)=Q$
\item[$\Leftrightarrow$] $\forall Q \in M_n(\cl T)^+, \exists P
\in M_n(\cl S), P + \ker \Phi_n \in M_n(\cl S \slash \ker \Phi)^+$
and $\widetilde \Phi_n(P + \ker \Phi_n)=Q$
\item[$\Leftrightarrow$] the induced map ${\widetilde \Phi} : \cl
S \slash \ker \Phi \to \cl T$ is a complete order isomorphism.
\end{enumerate}
\end{proof}

Recall that for operator spaces $V$ and $W$, the linear map $\Phi
: V \to W$ is called a complete quotient map if $\Phi_n : M_n(V)
\to M_n(W)$ is a quotient map for each $n \in \mathbb N$, that is,
$\Phi_n$ maps the open unit ball of $M_n(V)$ onto the open unit
ball of $M_n(W)$.

\begin{prop}\label{quotient map}
For operator systems $\cl S$ and $\cl T$, we suppose that $\Phi :
\cl S \to \cl T$ is a unital completely positive surjection. If
$\Phi : \cl S \to \cl T$ is a complete quotient map, then it is a
complete order quotient map.
\end{prop}

\begin{proof}
We choose $Q \in M_n(\cl T)^+$ and $\varepsilon >0$. We then have
$$-{1 \over 2} \|Q\| I_n \otimes 1_{\cl T} \le Q - {1 \over 2}\|Q\| I_n \otimes 1_{\cl T} \le {1 \over
2}\|Q\| I_n \otimes 1_{\cl T}.$$ There exists an element $P$ in
$M_n(\cl S)$ such that $$\Phi_n (P) = Q - {1 \over 2} \|Q\|I_n
\otimes 1_{\cl T} \qquad \text{and} \qquad \|P\| \le {1 \over
2}\|Q\| + \varepsilon.$$ By considering $(P+P^*)\slash 2$ instead,
we may assume that $P$ is self-adjoint. It follows that
$$P+{1 \over 2}\|Q\| I_n \otimes 1_{\cl S} + \varepsilon I_n \otimes 1_{\cl S} \in M_n(\cl S)^+ \qquad
\text{and} \qquad \Phi_n (P+ {1 \over 2}\|Q\| I_n \otimes 1_{\cl
S})=Q.$$
\end{proof}

The following theorem says that the maximal tensor product is
projective functorially in the category of operator systems.

\begin{thm}\label{projective}
For operator systems $\cl S_1, \cl S_2, \cl T$ and a complete
order quotient map $\Phi : \cl S_1 \to \cl S_2$, the linear map
$\Phi \otimes {\rm id}_{\cl T} : {\cl S}_1 \otimes_{\max} \cl T
\to \cl S_2 \otimes_{\max} \cl T$ is a complete order quotient
map.
\end{thm}

\begin{proof}
We choose an element $z$ in $M_n (\cl S_2 \otimes_{\max} \cl T)^+$
and $\varepsilon
>0$. Then we can write $$z + \varepsilon I_n \otimes 1_{\cl S_2} \otimes 1_{\cl T}
= \alpha P_2 \otimes Q \alpha^*,\qquad P_2 \in M_p(\cl S_2)^+, Q
\in M_q(\cl T)^+, \alpha \in M_{n, pq}.$$ We may assume that
$\|Q\|, \|\alpha \| \le 1$. There exists an element $P_1$ in
$M_p(\cl S_1)$ such that $$\Phi_p (P_1) = P_2 \qquad \text{and}
\qquad P_1 + \varepsilon I_p \otimes 1_{\cl S_1} \in M_p(\cl
S_1)^+.$$ It follows that
$$(\Phi \otimes {\rm id}_{\cl T})_n (\alpha P_1 \otimes Q \alpha^* -
\varepsilon I_n \otimes 1_{\cl S_1} \otimes 1_{\cl T}) = \alpha
P_2 \otimes Q \alpha^* - \varepsilon I_n \otimes 1_{\cl S_2}
\otimes 1_{\cl T} = z$$ and
$$\begin{aligned} & (\alpha P_1 \otimes Q \alpha^* - \varepsilon I_n \otimes
1_{\cl S_1} \otimes 1_{\cl T}) + 2 \varepsilon I_n
\otimes 1_{\cl S_1} \otimes 1_{\cl T} \\
 = & \alpha (P_1 + \varepsilon I_p \otimes 1_{\cl S_1}) \otimes Q \alpha^* +
 (\varepsilon I_n \otimes 1_{\cl S_1} \otimes 1_{\cl T} - \varepsilon \alpha ((I_p
 \otimes 1_{\cl S_1}) \otimes Q) \alpha^*) \\ \in &
 M_n(\cl S_1 \otimes_{\max} \cl T)^+.\end{aligned}$$
\end{proof}

Suppose that we are given an operator system $\cl S$ and a unital
$C^*$-algebra $\cl A$ such that $\cl S$ is an $\cl A$-bimodule.
Moreover, we assume that $1_{\cl A} \cdot s =s$ for $s \in \cl S$.
We call such an $\cl S$ an operator $\cl A$-system \cite[Chapter
15]{Pa} provided that $a \cdot 1_S = 1_S \cdot a$ and
$$[a_{i,j}] \cdot [s_{i,j}] \cdot [a_{i,j}]^* = [\sum_{k,l=1}^n
a_{i,k} \cdot s_{k,l} \cdot a_{j,l}^*] \in M_n(\cl S)^+, \qquad
[a_{i,j}] \in M_n(\cl A), [s_{i,j}] \in M_n(\cl S)^+.$$ The
maximal tensor product $\cl A \otimes_{\max} \cl S$ is an operator
$\cl A$-system \cite[Theorem 6.7]{KPTT1}.

The converse of Proposition \ref{quotient map} does not hold in
general since the operator space structure induced by the operator
system quotient by a kernel can be different from the operator
space quotient by it \cite[Example 4.4]{KPTT2}. However, the
converse holds in the following special situation. Although the
following theorem overlaps with \cite[Corollary 5.15]{KPTT2}, we
include it here because the proof is so elementary.

\begin{thm}\label{quotient}
Suppose that $\cl S$ is an operator system and $\cl A$ is a unital
$C^*$-algebra with its norm closed ideal $\cl I$. Then the
canonical map
$$\pi \otimes {\rm id}_{\cl S} : \cl A \otimes_{\max}
\cl S \to \cl A \slash \cl I \otimes_{\max} \cl S$$ is a complete
quotient map.
\end{thm}

\begin{proof}
By the nuclearity of matrix algebras, it is sufficient to show
that the canonical map $\pi \otimes {\rm id}_{\cl S} : \cl A
\otimes_{\max} \cl S \to \cl A \slash \cl I \otimes_{\max} \cl S$
is a quotient map. We suppose that
$$\|(\pi \otimes {\rm id}_{\cl S})(z)\|_{\cl A \slash
\cl I \otimes_{\max} \cl S} <1$$ for some $z \in \cl A \otimes \cl
S$. Then we have
$$\begin{pmatrix} (1-\varepsilon)1_{\cl A \slash \cl I} \otimes 1_{\cl S} & (\pi
\otimes {\rm id}_{\cl S})(z) \\ (\pi \otimes {\rm id}_{\cl
S})(z)^* & (1-\varepsilon) 1_{\cl A \slash \cl I} \otimes 1_{\cl
S}
\end{pmatrix} \in M_2(A \slash \cl I \otimes_{\max} \cl S)^+$$ for
$\varepsilon = 1 - \|(\pi \otimes {\rm id}_{\cl S})(z)\|_{\cl A
\slash \cl I \otimes_{\max} \cl S}$. By Theorem \ref{projective},
we can
find an element $\begin{pmatrix} \omega_{11} & \omega_{12} \\
\omega_{21} & \omega_{22} \end{pmatrix}$ in $M_2(\cl I \otimes \cl
S)$ such that
$$\begin{pmatrix} 1_{\cl A} \otimes 1_{\cl S} + \omega_{11} & z +
\omega_{12} \\ z^* + \omega_{21} & 1_{\cl A} \otimes 1_{\cl S} +
\omega_{22}
\end{pmatrix} \in M_2(\cl A \otimes_{\max} \cl S)^+.$$
Since we have $(\cl I \otimes \cl S)_h = \cl I_h \otimes \cl S_h$
\cite{CE}, $\omega_{11}$ can be written as $\omega_{11} =
\sum_{i=1}^n a_i \otimes s_i$ for $a_i \in I_h$ and $s_i \in \cl
S_h$. Let $a_i = a_i^+ - a_i^-$ for $1 \le i \le n$ and $a_i^+,
a_i^- \in \cl I^+$. We have
$$\omega_{11} = \sum_{i=1}^n a_i^+ \otimes s_i + a_i^- \otimes
(-s_i) \le \sum_{i=1}^n \|s_i\|(a_i^++a_i^-) \otimes 1_{\cl S}.$$
Hence, we can find elements $a$ and $d$ in $\cl I^+$ such that
$$\begin{pmatrix} (1_{\cl A} + a) \otimes 1_{\cl S} & z + \omega_{12} \\ z^* +
\omega_{21} & (1_{\cl A} + d)\otimes 1_{\cl S}
\end{pmatrix} \in M_2(\cl A \otimes_{\max} \cl S)^+.$$ Since $\cl A
\otimes_{\max} \cl S$ is an operator $\cl A$-system \cite[Theorem
6.7]{KPTT1}, we have
$$\begin{aligned} & \begin{pmatrix} 1_{\cl A} \otimes 1_{\cl S} & (1_{\cl A}+a)^{-{1 \over
2}} \cdot (z+\omega_{12}) \cdot (1_{\cl A}+d)^{-{1 \over 2}} \\
(1_{\cl A}+d)^{-{1 \over 2}} \cdot (z^*+\omega_{21}) \cdot (1_{\cl
A}+a)^{-{1 \over 2}} & 1_{\cl A} \otimes 1_{\cl S}
\end{pmatrix} \\ = & \begin{pmatrix} (1_{\cl A} + a)^{-{1 \over 2}} & 0 \\ 0
& (1_{\cl A} + d)^{-{1 \over 2}} \end{pmatrix} \cdot
\begin{pmatrix} (1_{\cl A} + a) \otimes
1_{\cl S} & z + \omega_{12} \\ z^* + \omega_{21} & (1_{\cl A} +
d)\otimes 1_{\cl S}
\end{pmatrix} \cdot \begin{pmatrix} (1_{\cl A} + a)^{-{1 \over 2}} & 0 \\ 0
& (1_{\cl A} + d)^{-{1 \over 2}} \end{pmatrix} \\ \in & \ M_2(\cl
A \otimes_{\max} \cl S)^+. \end{aligned}$$ It follows that
$$\| (1_{\cl A}+a)^{-{1 \over 2}} \cdot (z+\omega_{12}) \cdot
(1_{\cl A}+d)^{-{1 \over 2}} \|_{\cl A \otimes_{\max} \cl S} \le
1$$ and $$(\pi \otimes {\rm id}_{\cl S}) ((1_{\cl A}+a)^{-{1 \over
2}} \cdot (z+\omega_{12}) \cdot (1_{\cl A}+d)^{-{1 \over 2}}) =
(\pi \otimes {\rm id}_{\cl S})(z).$$
\end{proof}

\section{The equivalence of the $({\rm el},\max)$-nuclearity and the WEP}

As we have seen in the previous section, the maximal tensor
product is projective. However, the maximal tensor product need
not be injective. This misbehavior distinguishes a class which is
called $({\rm el},\max)$-nuclear operator systems \cite{KPTT2}. In
the category of $C^*$-algebras, Lance characterized this tensorial
property by the factorization property for the inclusion map into
the second dual through $B(H)$ \cite{L1,L2}. In \cite{KPTT2}, it
is proved that this factorization property for an operator system
implies its $({\rm el},\max)$-nuclearity and it is asked whether
the converse holds. In this section, we answer this question in
the affirmative, independent of Lance's original theorem.

\begin{thm}\label{WEP}
Let $\cl S$ be an operator system. The following are equivalent:
\begin{enumerate}
\item[(i)] we have $$\cl S \otimes_{\max} \cl T \subset \cl S_2
\otimes_{\max} \cl T$$ for any inclusion $\cl S \subset \cl S_2$
and any operator system $\cl T$; \item[(ii)] we have $$\cl S
\otimes_{\max} E \subset \cl S_2 \otimes_{\max} E$$ for any
inclusion $\cl S \subset \cl S_2$ and any finite dimensional
operator system $E$; \item[(iii)] we have $$\cl S \otimes_{\max} E
\subset B(H) \otimes_{\max} E$$ for any inclusion $\cl S \subset
B(H)$ and any finite dimensional operator system $E$; \item[(iv)]
there exist unital completely positive maps $\Phi : \cl S \to
B(H)$ and $\Psi : B(H) \to \cl S^{**}$ such that $\Psi \circ \Phi
= \iota$ for the canonical inclusion $\iota : \cl S
\hookrightarrow \cl S^{**}$.
$$\xymatrix{\cl S \ar@{^{(}->}[rr]^\iota \ar[dr]_\Phi && \cl S^{**} \\
& B(H) \ar[ur]_\Psi &}$$
\end{enumerate}
\end{thm}

\begin{proof}
Clearly, (i) implies (ii) and (ii) implies (iii). The direction
$\rm (iv) \Rightarrow (i)$ follows from \cite[Theorem 6.7]{KPTT2}.

$\rm (iii)\Rightarrow (iv).$ Suppose that an operator system $\cl
S$ acts on a Hilbert space $H$. Considering the bidual of the
inclusion $\cl S \subset B(H)$ and the universal representation of
$B(H)$, we may assume that the inclusions $\cl S \subset \cl
S^{**} \subset B(H)$ are given such that the second inclusion is
weak$^*$-WOT homeomorphic. Let $\{ P_\lambda \}$ be the family of
projections on the finite dimensional subspaces of $H$ directed by
the inclusions of their ranges. We put $$E_\lambda = P_\lambda
\mathcal S P_\lambda \qquad \text{and} \qquad \Phi_\lambda :=
P_\lambda \cdot P_\lambda : \mathcal S \to E_\lambda.$$ Since each
$E_\lambda$ is a finite dimensional operator system, there exists
a non-canonical Archimedean order unit on its dual space
$E_\lambda^*$ \cite[Corollary~4.5]{CE}. In other words, the dual
space $E^*_\lambda$ is an operator system. By \cite[Lemma
5.7]{KPTT1}, a functional $\varphi_\lambda$ on $\mathcal S
\otimes_{\max} E_\lambda^*$ corresponding to the compression
$\Phi_\lambda : \mathcal S \to E_\lambda$ in the standard way is
positive. By assumption, $\cl S \otimes_{\max} E_\lambda^*$ is an
operator subsystem of $B(H) \otimes_{\max} E_\lambda^*$. By
Krein's theorem, $\varphi_\lambda$ extends to a positive
functional $\psi_\lambda$ on $B(H) \otimes_{\max} E_\lambda^*$.
$$\xymatrix{B(H) \otimes_{\max} E_\lambda^* \ar[drr]^-{\psi_\lambda} && \\
\cl S \otimes_{\max} E_\lambda^* \ar@{^{(}->}[u]
\ar[rr]_-{\varphi_\lambda} && \bb C}$$ Let $\Psi_\lambda : B(H)
\to E_\lambda$ be a completely positive map corresponding to
$\psi_\lambda$ in the standard way. Since $\Phi_\lambda =
\Psi_\lambda \circ \iota$ for the inclusion $\iota : \cl S \subset
B(H)$, $\Psi_\lambda$ is a unital completely positive map. We take
a state $\omega_\lambda$ on $E_\lambda$ and define a unital
completely positive map $\theta_\lambda : E_\lambda \to B(H)$ by
$$\theta_\lambda (x) = x+ \omega_\lambda(x)(I-P_\lambda), \qquad x \in
E_\lambda.$$ Let $\Psi : B(H) \to B(H)$ be a point-weak$^*$
cluster point of $\{ \theta_\lambda \circ \Psi_\lambda \}$. We may
assume that $\theta_\lambda \circ \Psi_\lambda$ converges to
$\Psi$ in the point-weak$^*$ topology. For $\xi, \eta \in H$ and
$x \in \cl S$, we have $$\langle \Psi (x) \xi, \eta \rangle =
\lim_\lambda \langle (P_\lambda x P_\lambda + \omega_\lambda
(P_\lambda x P_\lambda)(I-P_\lambda))\xi , \eta \rangle = \langle
x \xi, \eta \rangle.$$ It follows that $\Psi|_{\cl S} = \iota$.
Since $\Phi_\lambda$ is surjective and $\Psi_\lambda$ is an
extension of $\Phi_\lambda$, for each $x \in B(H)$ there exists an
element $x_\lambda$ in $\cl S$ such that $\Psi_\lambda(x) =
\Phi_\lambda (x_\lambda) = P_\lambda x_\lambda P_\lambda$. For
$\xi, \eta \in H$ and $x \in B(H)$, we have
$$\langle \Psi (x) \xi, \eta \rangle = \lim_\lambda \langle
(P_\lambda x_\lambda P_\lambda + \omega_\lambda (P_\lambda
x_\lambda P_\lambda)(I-P_\lambda))\xi , \eta \rangle =
\lim_\lambda \langle x_\lambda \xi, \eta \rangle.$$ It follows
that $\Psi(x)$ belongs to the WOT-closure of $\cl S$ which
coincides with $\cl S^{**}$ because the inclusion $\cl S^{**}
\subset B(H)$ is weak$^*$-WOT homeomorphic.
\end{proof}

As a corollary, we deduce the following theorem of Lance. The
proof is similar to that of \cite[Corollary 3.3]{HP}.

\begin{cor}[Lance] Let $\cl A$ be a unital $C^*$-algebra.
Then we have $\cl A \otimes_{\max} \cl B \subset \cl A_2
\otimes_{\max} \cl B$ for any inclusion $\cl A \subset \cl A_2$
and any unital $C^*$-algebra $\cl B$ if and only if $\cl A$ has
the weak expectation property.
\end{cor}

\begin{proof}
By Theorem \ref{WEP}, it will be enough to prove that if $\cl A
\otimes_{\max} \cl B \subset B(H) \otimes_{\max} \cl B$ for any
unital $C^*$-algebra $\cl B$, then $\cl A \otimes_{\max} \cl T
\subset B(H) \otimes_{\max} \cl T$ for any operator system $\cl
T$. Due to \cite[Theorem~6.4]{KPTT1} and
\cite[Theorem~6.7]{KPTT1}, we obtain the following commutative
diagram, which yields the conclusion:
$$\xymatrix{\cl A \otimes_{\max}
\cl T \ar@{=}[r] \ar[d] & \cl A \otimes_{\rm c} \cl
T \ar@{^{(}->}[r] & \cl A \otimes_{C^*\max} C^*_u(\cl T) \ar@{_{(}->}[d] \\
B(H) \otimes_{\max} \cl T \ar@{=}[r] & B(H) \otimes_{\rm c} \cl T
\ar@{^{(}->}[r] & B(H) \otimes_{C^*\max} C^*_u(\cl T)}$$
\end{proof}

Examples of nuclear operator systems which are not unitally
completely order isomorphic to any unital $C^*$-algebra have been
constructed in \cite{KW,HP}. These also provide examples of
operator systems with the WEP which are not unitally completely
order isomorphic to any unital $C^*$-algebra.

\section{Completely positive maps associated with maximal tensor products}

The order unit of the dual spaces of operator systems cannot be
considered in general. However, the dual spaces of finite
dimensional operator systems have a non-canonical Archimedean
order unit. This enables the duality between tensor products and
mapping spaces to work in the proofs of the Choi-Effros-Kirchberg
theorem for operator systems \cite{HP} and the Lance theorem for
operator systems in the previous section. Not only is the finite
dimensional assumption restrictive, but also the matrix order unit
norm on the dual spaces of finite dimensional operator systems is
irrelevant to the matrix norm given by the standard dual of
operator spaces. To get rid of the finite dimensional assumption
and to reflect the operator space dual norm, we apply Werner's
unitization of matrix ordered operator spaces~\cite{W} to the dual
spaces of operator systems.

Let $V$ be a matrix ordered operator space. We give the involution
and the matrix order on $V \oplus \mathbb C$ as follows:
\begin{enumerate}
\item $(x+a)^* = x^*+ \bar{a}, \qquad x \in V, a \in \mathbb C$
\item $X + A \in M_n(V \oplus \mathbb C)^+$ iff $$A \in M_n^+
\quad \text{and} \quad \varphi((A+\varepsilon I_n)^{-{1 \over 2}}
X (A+\varepsilon I_n)^{-{1 \over 2}}) \ge -1$$ for any
$\varepsilon>0$ and any positive contractive functional $\varphi$
on $M_n(V)$.
\end{enumerate}
We denote by $\widetilde V$ the space $V \oplus \mathbb C$ with
the above involution and matrix order and call it the unitization
of $V$ \cite[Definition 4.7]{W}. The unitization $\widetilde V$ of
a matrix ordered operator space $V$ is an operator system and the
canonical inclusion $\iota : V \hookrightarrow \widetilde V$ is a
completely contractive complete order isomorphism onto its range
\cite[Lemma 4.8]{W}. However it need not be completely isomorphic.
We apply Werner's unitization of matrix ordered operator spaces to
the dual spaces of operator systems. In this case, the canonical
inclusion $\iota : \cl S^* \hookrightarrow \widetilde{\cl S^*}$ is
2-completely isomorphic \cite{Kar, H, KPTT1}.

\begin{lem}\label{unitization}
Suppose that $\cl S$ is an operator system and $\widetilde{\cl
S^*}$ is the unitization of the dual space $\cl S^*$. Let $f : \cl
S \to M_n$ be a self-adjoint linear map  and $A \in M_n^+$. Then
$f+A$ belongs to $M_n(\widetilde{\cl S^*})^+$ if and only if we
have $f_m(x) \ge - I_m \otimes A$ for all $m \in \mathbb N$ and $x
\in M_m(\cl S)_1^+$.
\end{lem}

\begin{proof}
$\Rightarrow )$ The element $f+A$ belongs to $M_n(\widetilde{\cl
S^*})^+$ if and only if
$$\varphi((A+\varepsilon I_n )^{-{1 \over 2}} f (A+\varepsilon
I_n)^{-{1 \over 2}}) \ge -1$$ for any $\varepsilon
>0$ and any positive contractive functional $\varphi$ on
$M_n(\cl S^*)$. For $x \in M_m(\cl S)_1^+$ and $\xi \in
(\ell^2_{mn})_1$, the map $$\varphi_{x, \xi} : f \in M_n(\cl
S^*)=CB(\cl S,M_n) \mapsto \langle f_m(x) \xi | \xi \rangle \in
\mathbb C$$ is a positive contractive functional on $M_n(\cl
S^*)$. It follows that $$\begin{aligned} & \langle (I_m \otimes
(A+\varepsilon I_n)^{-{1 \over 2}})f_m(x)(I_m \otimes
(A+\varepsilon I_n)^{-{1 \over 2}}) \xi | \xi \rangle \\ = &
\varphi_{x,\xi}((A+\varepsilon I_n )^{-{1 \over 2}}\ f\
(A+\varepsilon I_n)^{-{1 \over 2}}) \\ \ge & -1.
\end{aligned}$$ Hence, we have $f_m(x) \ge - I_m
\otimes A$.

\vskip 1pc

$\Leftarrow )$ Put $\Omega = \{ \varphi_{x,\xi} : m \in \bb N, x
\in M_m(\cl S)_1^+, \xi \in (\ell^2_{mn})_1 \}$ where
$\varphi_{x,\xi}$ defined as above. Let $\Gamma_1$ be a
weak$^*$-closed convex hull of $\Omega$ and $\Gamma_2$ a
weak$^*$-closed cone generated by $\Omega$. We want $\Gamma_1 =
(M_n(\cl S^*)^*)^+_1$. Here, $(M_n(\cl S^*)^*)^+_1$ denotes the
set of positive contractive functionals on $M_n(\cl S^*)$. If this
were not the case, we could choose $\varphi_0 \in (M_n(\cl
S^*)^*)^+_1 \slash \Gamma_1$. By the Krein-Smulian
theorem~\cite[Theorem 5.12.1]{C}, we have $\bb R^+ \cdot \Gamma_1
= \Gamma_2$, thus $\varphi_0$ does not belong to $\Gamma_2$. By
the Hahn-Banach separation theorem, there exists $f_0 \in M_n(\cl
S^*)_{sa}$ which separates $\Gamma_2$ and $\varphi_0$ strictly. We
have $\Gamma_2 (f_0) = \{ 0 \}$ or $[0,\infty)$ or $(-\infty, 0]$.
If $\varphi_{x,\xi} (f_0) = 0$ for all $\varphi_{x,\xi} \in
\Omega$, then we have $f_0=0$. We may assume that
$\Gamma_2(f_0)=[0,\infty)$. Then we have $f_0 \in M_n(\cl S^*)^+$
which is a contradiction since $f_0$ separates $\Gamma_2$ and
$\varphi_0$ strictly. The conclusion follows from $\Gamma_1 =
(M_n(\cl S^*)^*)^+_1$ and
$$\varphi_{x,\xi}((A+\varepsilon I_n )^{-{1 \over 2}}\ f\
(A+\varepsilon I_n)^{-{1 \over 2}}) \ge -1.$$
\end{proof}

\begin{prop}\label{min}
Suppose that $\Phi : \cl S \to \cl T$ is a finite rank map for
operator systems $\cl S$ and $\cl T$. Then $\Phi$ is completely
positive if and only if it belongs to $(\widetilde{\cl S^*}
\otimes_{\min} \cl T)^+$.
\end{prop}

\begin{proof}
$\Rightarrow )$ The finite rank map $\Phi$ can be regarded as an
element in $\cl S^* \otimes \cl T \subset \widetilde{\cl S^*}
\otimes \cl T$. For a positive element $x$ in $M_n(\cl S)$, the
evaluation map ${\rm ev}_x : {\cl S}^* \to M_n$ defined by
$${\rm ev}_x(f) = f_n(x), \qquad f \in \cl S^*$$ is
completely positive. For a completely positive map $g : \cl T \to
M_m$, we have
$$({\rm ev}_x \otimes g)(\Phi) = (g \circ \Phi)_n(x) \in
M_{mn}^+, \qquad ({\rm ev}_x \otimes g : \cl S^* \otimes \cl T \to
M_{mn}).$$ For a completely positive map $f : \widetilde{\cl S^*}
\to M_n$, its restriction $f|_{\cl S^*}$ belongs to $CP(\cl S^*,
M_n) = M_n(\cl S^{**})^+$. It is the point-norm limit of
evaluation maps ${\rm ev}_x$ for $x \in M_n(\cl S)^+$. Thus, we
have $$(f \otimes g)(\Phi) = (f|_{\cl S^*} \otimes g)(\Phi) \in
M_{mn}^+.$$ In other words, $\Phi$ belongs to $(\widetilde{\cl
S^*} \otimes_{\min} \cl T)^+$.

\vskip 1pc

$\Leftarrow )$ For $x \in M_n(\cl S)^+_1$, we define an evaluation
map ${\rm ev}_x : \widetilde{\cl S^*} \to M_n$ by
$${\rm ev}_x (f+\lambda) = f_n(x) + \lambda I_n.$$ By Lemma
\ref{unitization}, the evaluation map ${\rm ev}_x : \widetilde{\cl
S^*} \to M_n$ is completely positive.

From $$(g \circ \Phi)_n(x) = ({\rm ev}_x \otimes g)(\Phi) \in
M_{mn}^+$$ we see that $\Phi : \cl S \to \cl T$ is completely
positive.
\end{proof}

By Proposition \ref{min}, finite rank completely positive maps
from $\cl S$ to $\cl T$ correspond to the elements in $\cl S^*
\otimes \cl T \cap (\widetilde{\cl S^*} \otimes_{\min} \cl T)^+$.
Our next goal is to study completely positive maps corresponding
to the elements in $\cl S^* \otimes \cl T \cap (\widetilde{\cl
S^*} \otimes_{\max} \cl T)^+$.
\begin{thm}\label{nuclear}
Suppose that $\Phi : \cl S \to \cl T$ is a finite rank map. Then
$\Phi$ belongs to $(\widetilde{\cl S^*} \otimes_{\max} \cl T)^+$
if and only if for any $\varepsilon>0$, there exist a
factorization $\Phi = \psi \circ \varphi$ and a positive
semidefinite matrix $A \in M_p$ such that $\varphi : \cl S \to
M_p$ is a self-adjoint map with $\varphi_m(x) \ge -I_m \otimes A$
for all $m \in \mathbb N, x \in M_m(\cl S)_1^+$ and $\psi : M_p
\to \cl T$ is a completely positive map with $\psi(A) =
\varepsilon 1_{\cl T}$.
$$\xymatrix{\cl S \ar[rr]^\Phi \ar[dr]_\varphi && \cl T \\
& M_p \ar[ur]_\psi &}$$
\end{thm}

\begin{proof}
$\Rightarrow )$ Let $\Phi \in (\widetilde{\cl S^*} \otimes_{\max}
\cl T)^+$. For any $\varepsilon > 0$, we can write $$\Phi +
\varepsilon 1_{\widetilde{\cl S^*}} \otimes 1_{\cl T} = \alpha
((\varphi+A \cdot 1_{\widetilde{\cl S^*}}) \otimes Q) \alpha^*$$
for $\alpha \in M_{1,pq}, \varphi \in M_p(\cl S^*), A \in M_p, Q
\in M_q(\cl T)^+$ and $\varphi + A \cdot 1_{\widetilde{\cl S^*}}
\in M_p(\widetilde{S^*})^+$. It follows that $$\Phi(x) = \alpha
\varphi(x) \otimes Q \alpha^* \quad \text{and} \quad \varepsilon
1_{\cl T} = \alpha (A \otimes Q) \alpha^*.$$ By Lemma
\ref{unitization}, we have
$$\varphi_n(x) \ge -I_n \otimes A,\qquad n \in \mathbb N, x \in
M_n(\cl S)^+_1.$$ We define a completely positive map $\psi : M_p
\to \cl T$ by $\psi(B)=\alpha (B \otimes Q) \alpha^*$. Then we
have $$\Phi(x) = \alpha \varphi(x) \otimes Q \alpha^* = \psi
(\varphi (x)) \quad \text{and} \quad \psi(A) = \alpha (A \otimes
Q) \alpha^* = \varepsilon 1_{\cl T}.$$

\vskip 1pc

$\Leftarrow )$ We put $$Q = [\psi(e_{i,j})]_{1 \le i,j \le p} \in
M_p (\cl T)^+ \quad \text{and} \quad \alpha = [e_1, \cdots, e_p]
\in M_{1, p^2}.$$ Because
$$\psi([b_{i,j}]) = \psi(\sum_{i,j=1}^p b_{i,j} e_{i,j}) =
\sum_{i,j=1}^p b_{i,j} Q_{i,j} = \alpha ([b_{i,j}] \otimes Q)
\alpha^*,$$ we can write $$\Phi(x) = \alpha \varphi(x) \otimes Q
\alpha^* \quad \text{and} \quad \varepsilon 1_{\cl T} = \alpha (A
\otimes Q) \alpha^*.$$ By Lemma \ref{unitization}, we have
$\varphi+A \in M_n(\widetilde{\cl S^*})^+$. It follows that
$$\Phi + \varepsilon 1_{\widetilde{\cl S^*}} \otimes
1_{\cl T} = \alpha ((\varphi+A \cdot 1_{\widetilde{\cl S^*}})
\otimes Q) \alpha^* \in (\widetilde{\cl S^*} \otimes_{\max} \cl
T)^+$$
\end{proof}

Let $\Phi : \cl S \to \cl T$ be a completely positive map
factoring through a matrix algebra in a completely positive way.
By the proof of Theorem \ref{nuclear}, $\Phi$ corresponds to an
element in the subcone $$\{ \alpha (\varphi \otimes Q) \alpha^* :
\alpha \in M_{1,pq}, \varphi \in M_p(\cl S^*)^+, Q \in M_q(\cl
T)^+ \}$$ of the cone $\cl S^* \otimes \cl T \cap (\widetilde{\cl
S^*} \otimes_{\max} \cl T)^+$.

\begin{thm}
Suppose that $\Phi : \cl S \to \cl T$ is a completely positive map
for operator systems $\cl S$ and $\cl T$. The map
$${\rm id}_{\cl R} \otimes \Phi : \cl R \otimes_{\min}
\cl S \to \cl R \otimes_{\max} \cl T$$ is completely positive for
any operator system $\cl R$ if and only if we have
$$\Phi|_E \in (\widetilde{E^*} \otimes_{\max} \cl T)^+$$ for
any finite dimensional operator subsystem $E$ of $\cl S$.
\end{thm}

\begin{proof}
$\Rightarrow )$ By Proposition \ref{min}, we can regard the
inclusion $\iota : E \subset \cl S$ as an element in
$(\widetilde{E^*} \otimes_{\min} \cl S)^+$. By assumption, we see
that $\Phi|_E = ({\rm id}_{\widetilde{E^*}} \otimes \Phi)(\iota)$
belongs to $(\widetilde{E^*} \otimes_{\max} \cl T)^+$.

\vskip 1pc

$\Leftarrow )$ We choose an element $$z = \sum_{i=1}^n x_i \otimes
y_i \in (\cl R \otimes_{\min} \cl S)^+_1.$$ Let $E$ be a finite
dimensional operator subsystem of $\cl S$ which contains $\{y_i :
1 \le i \le n \}$. By Theorem \ref{nuclear}, there exist a
factorization $\Phi|_E = \psi \circ \varphi$ and a positive
semidefinite matrix $A \in M_n$ such that $\varphi : E \to M_n$ is
a self-adjoint map with $\varphi_m(x) \ge -I_m \otimes A$ for all
$m \in \mathbb N, x \in M_m(\cl S)_1^+$ and $\psi : M_n \to \cl T$
is a completely positive map with $\psi(A) = \varepsilon 1_{\cl
T}$. Let $\cl R$ be a concrete operator system acting on a Hilbert
space $H$ and $P$ the projection onto the finite dimensional
subspace of $H$. The compression $P \cl R P$ is the operator
subsystem of a matrix algebra $M_p$ for $p = {\rm rank} P$. From
the commutative diagram
$$\xymatrix{M_p \otimes_{\min} E \ar[rr]^{{\rm id}_{M_p}
\otimes \varphi}
&& M_p \otimes_{\min} M_n \\
P \cl R P \otimes_{\min} E \ar@{^{(}->}[u] \ar[rr]^{{\rm id}_{P
\cl R P} \otimes \varphi} && P \cl R P \otimes_{\min} M_n
\ar@{^{(}->}[u]}$$ we see that
$$({\rm id}_{\cl R} \otimes \varphi) (z) \ge -
1_{\cl R} \otimes A.$$ From the commutative diagram
$$\xymatrix{\cl R \otimes_{\min} \cl S \ar[rr]^{{\rm id}_{\cl R}
\otimes \Phi} && \cl R \otimes_{\max} \cl T \\
\cl R \otimes_{\min} E \ar@{^{(}->}[u] \ar[rr]^-{{\rm id}_{\cl R}
\otimes \varphi} && \cl R \otimes_{\min} M_n = \cl R
\otimes_{\max} M_n \ar[u]_{{\rm id}_{\cl R} \otimes \psi}}$$ we
also see that
$$\begin{aligned} ({\rm id}_{\cl R} \otimes \Phi)(z) & =
({\rm id}_{\cl R} \otimes \psi) \circ ({\rm id}_{\cl R} \otimes
\varphi)(z) \\ & \ge -\varepsilon ({\rm id}_{\cl R} \otimes
\psi)(1_{\cl R} \otimes A) \\ & = - \varepsilon 1_{\cl R} \otimes
1_{\cl T} \end{aligned}$$ in $\ \cl R \otimes_{\max} \cl T$. Since
the choice of $\varepsilon>0$ is arbitrary, we conclude that the
map
$${\rm id}_{\cl R} \otimes \Phi : \cl R \otimes_{\min}
\cl S \to \cl R \otimes_{\max} \cl T$$ is positive.

\end{proof}

\section{Duality}

We establish the duality between complete order embeddings and
complete order quotient maps.

\begin{thm}
Suppose that $\cl S$ and $\cl T$ are operator systems with
complete norms and $\Phi : \cl S \to \cl T$ is a unital completely
positive surjection. Then $\Phi : \cl S \to \cl T$ is a complete
order quotient map if and only if its dual map $\Phi^* : \cl T^*
\to \cl S^*$ is a complete order embedding.
\end{thm}

\begin{proof}
$\Rightarrow )$ Let $\Phi_n^* (f) \in M_n(\cl S^*)^+ = CP(\cl S,
M_n)$ for $f \in M_n(\cl T^*)$. We choose a positive element $y$
in $M_m(\cl T)$. For any $\varepsilon>0$, there exists an element
$x$ in $M_m(\cl S)$ such that $$\Phi_m(x)=y \qquad \text{and}
\qquad x+\varepsilon I_m \otimes 1_{\cl S} \in M_m(\cl S)^+.$$ We
have $$(f \circ \Phi)_m(x) + \varepsilon (f \circ \Phi)_m(I_m
\otimes 1_{\cl S}) = (\Phi^*_n(f))_m(x + \varepsilon I_n \otimes
1_{\cl S}) \in M_{mn}^+.$$ Since the choice of $\varepsilon>0$ is
arbitrary, we have $$f_m(y) = (f \circ \Phi)_m(x) \in M_{mn}^+.$$
It follows that $f : \cl T \to M_n$ is completely positive.

\vskip 1pc

$\Leftarrow )$ We put $$C_n := \{ y \in M_n(\cl T) : \forall
\varepsilon >0, \exists x \in M_n(\cl S), x + \varepsilon I_n
\otimes 1_{\cl S} \in M_n(\cl S)^+ \ \text{and} \ \Phi_n(x)=y
\}.$$ For $y \in C_n$, we have $$y+\varepsilon I_n \otimes 1_{\cl
T} = \Phi_n (x + \varepsilon I_n \otimes 1_{\cl S}) \in M_n(\cl
T)^+,$$ thus $C_n \subset M_n(\cl T)^+$. The map $\Phi : \cl S \to
\cl T$ is a complete order quotient map if and only if $C_n =
M_n(\cl T)^+$ holds for all $n \in \mathbb N$. It is easy to check
that $C_n$ is a cone. We have the inclusions of the cones
$$\Phi_n(M_n(\cl S)^+) \subset C_n \subset M_n(\cl
T)^+.$$ Suppose that $y_k \in C_n$ converges to $y \in M_n(\cl
T)$. By the open mapping theorem, there exists $M>0$ such that $\|
{\widetilde \Phi}_n^{-1} : M_n (\cl T) \to M_n(\cl S) \slash {\rm
Ker} \Phi_n \| \le M$. We choose $y_{k_0}$ so that $\|y - y_{k_0}
\| < \varepsilon \slash M$. There exist $x, x'$ in $M_n(\cl S)$
such that
$$\Phi_n(x)=y_{k_0},\ x+\varepsilon I_n \otimes 1_{\cl S} \in
M_n(\cl S)^+ \quad \text{and} \quad \Phi_n(x') = y-y_{k_0},\
\|x'\| < \varepsilon.$$ Replacing $x'$ by $(x'+x'^*) \slash 2$, we
may assume that $x'$ is self-adjoint. It follows that $$y=y_{k_0}
+ (y-y_{k_0}) = \Phi_n(x+x') \quad \text{and} \quad
x+x'+2\varepsilon I_n \otimes 1_{\cl S} \in M_n(\cl S)^+.$$ Hence,
the cone $C_n$ is closed. We assume $C_n \subsetneq M_n(\cl T)^+$
and choose $y_0 \in M_n(\cl T)^+ \slash C_n$. By the Hahn-Banach
separation theorem, there exists a self-adjoint functional $f$ on
$M_n(\cl T)$ such that $f(C_n)=[0,\infty)$ and $f(y_0)<0$. Even
though the functional $f : M_n(\cl T) \to \mathbb C$ is not
positive, we have $\Phi_n^* (f)(x) = f \circ \Phi_n(x) \ge 0$ for
all $x \in M_n(\cl S)^+$ because $\Phi_n(M_n(\cl S)^+)$ is a
subcone of $C_n$. Hence, we see that the dual map $\Phi^* : \cl
T^* \to \cl S^*$ is not a complete order embedding.
\end{proof}

\begin{lem}\label{lemma}
Suppose that $f : \cl S \to M_n$ is a self-adjoint linear map for
an operator system $\cl S$. Then we have $f_m(x) \ge -I_{mn}$ for
all $m \in \mathbb N, x \in M_m(S)_1^+$ if and only if the
self-adjoint linear map $\widetilde f : \cl S \to M_n$ defined by
$$\widetilde f (x) = (f(1_{\cl S}) + 2 I_n)^{-{1 \over 2}} f(x) (f(1_{\cl S})
+ 2 I_n)^{-{1 \over 2}}$$ is completely contractive.
\end{lem}

\begin{proof}
$\Rightarrow )$ We choose a contractive element $x$ in
$M_m(\cl S)$. Then we have $$0 \le \begin{pmatrix} I_m \otimes 1_{\cl S} & x \\
x^* & I_m \otimes 1_{\cl S}
\end{pmatrix} \le 2 \begin{pmatrix} I_m \otimes 1_{\cl S} & 0 \\ 0 & I_m \otimes 1_{\cl S}
\end{pmatrix}.$$ By assumption, we have $$ 0 \le
\begin{pmatrix} I_m \otimes f(1_{\cl S}) +2
I_{mn} & f_m(x) \\ f_m(x)^* & I_m \otimes f(1_{\cl S}) +2 I_{mn}
\end{pmatrix}.$$ Multiplying both sides by $(I_m \otimes
f(1_{\cl S}) +2 I_{mn})^{-{1 \over 2}} \oplus (I_m \otimes
f(1_{\cl S}) +2 I_{mn})^{-{1 \over 2}}$ from the left and from the
right, we see that $\|{\widetilde f}_m (x)\| \le 1$.

\vskip 1pc

$\Leftarrow )$ We choose an element $x$ in $M_m(\cl S)^+_1$. Then
we have $$\|x-{1 \over 2} I_m \otimes 1_{\cl S} \| \le {1 \over
2}.$$ By assumption, we have
$$\begin{aligned} & \|(I_m \otimes f(1_{\cl S}) + 2
I_{mn})^{-{1 \over 2}} f_m(x-{1 \over 2} I_m \otimes 1_S) (I_m
\otimes f(1_{\cl S}) + 2 I_{mn})^{-{1 \over 2}}\| \\ = &
\|{\widetilde f}_m(x-{1 \over 2} I_m \otimes 1_{\cl S}) \|
\\ \le & {1 \over 2}, \end{aligned}$$ thus $$-{1 \over 2} (I_m \otimes f(1_{\cl S}) + 2
I_{mn}) \le f_m(x-{1 \over 2} I_m \otimes 1_{\cl S}).$$ It follows
that $f_m(x) \ge -I_{mn}$.
\end{proof}

\begin{lem} \label{extension}
Let $\cl S$ be an operator subsystem of an operator system $\cl T$
and $A$ a positive semidefinite $n \times n$ matrix. Suppose that
$f : \cl S \to M_n$ is a self-adjoint linear map satisfying
$f_m(x) \ge - I_m \otimes A$ for all $m \in \mathbb N, x \in
M_m(\cl S)_1^+$. Then $f : \cl S \to M_n$ extends to a
self-adjoint linear map $F : \cl T \to M_n$ satisfying $F_m(x) \ge
-I_m \otimes A$ for all $m \in \mathbb N, x \in M_m(\cl T)_1^+$.
\end{lem}

\begin{proof}
We define a self-adjoint linear map $g : \cl S \to M_n$ by
$$g(x) = (A+\varepsilon I_n)^{-{1 \over 2}}f(x)(A+\varepsilon
I_n)^{-{1 \over 2}}$$ for $0 < \varepsilon < 1$. Then we have
$$g_m(x) \ge -(I_m \otimes A +\varepsilon I_{mn})^{-{1 \over
2}}(I_m \otimes A) (I_m \otimes A +\varepsilon I_{mn})^{-{1 \over
2}} \ge -I_{mn}$$ for all $m \in \mathbb N, x \in M_m(\cl S)_1^+$.
By Lemma \ref{lemma}, the self-adjoint linear map $\widetilde g :
\cl S \to M_n$ defined by $$\widetilde g (x) = (g(1_{\cl S})+2
I_n)^{-{1 \over 2}} g(x) (g(1_{\cl S})+2 I_n)^{-{1 \over 2}}$$ is
completely contractive. By the Wittstock extension theorem,
$\widetilde g$ extends to a complete contraction $\widetilde G :
\cl T \to M_n$. By considering ${1 \over 2} (\widetilde
G+{\widetilde G}^*)$ instead, we may assume that $\widetilde G$ is
self-adjoint. We put $$G(x) = (g(1_{\cl S})+2 I_n)^{1 \over 2}
\widetilde G(x)(g(1_{\cl S})+2 I_n)^{1 \over 2}$$ and
$$F(x)=(A+\varepsilon I_n)^{1 \over 2} G(x)(A + \varepsilon
I_n)^{1 \over 2}.$$ Then $F : \cl T \to M_n$ (respectively, $G :
\cl T \to M_n$) is a self-adjoint extension of $f : \cl S \to M_n$
(respectively, $g : \cl S \to M_n$). The self-adjoint linear map
$\widetilde G$ can be written as
$$\widetilde G(x) = (G(1_{\cl S})+2 I_n)^{-{1 \over 2}} G(x) (G(1_{\cl S})+2
I_n)^{-{1 \over 2}}.$$ By using Lemma \ref{lemma} again, we see
that
$$F_m(x) = (I_m \otimes A+\varepsilon I_{mn})^{1 \over 2} G_m(x)(I_m \otimes A+\varepsilon
I_{mn})^{1 \over 2} \ge - I_m \otimes A - \varepsilon I_{mn}$$ for
all $m \in \mathbb N, x \in M_m(\cl T)_1^+$. The extension $F :
\cl T \to M_n$ depends on the choice of $\varepsilon$. However the
norm of $F$ is uniformly bounded as can be seen from
$$\begin{aligned} \|F\| & \le \|(A+\varepsilon I_n)^{1 \over
2}((A+\varepsilon I_n)^{-{1 \over 2}}f(1_{\cl
S})(A+\varepsilon I_n)^{-{1 \over 2}} + 2I_n)^{1 \over 2} \|^2 \\
& = \|f(1_{\cl S}) + 2(A+\varepsilon I_n)\| \\ & \le \|f(1_{\cl
S}) + 2A \| +2.\end{aligned}$$ Since the range space is finite
dimensional, we can consider the point-norm cluster point of $\{
F_\varepsilon : 0< \varepsilon <1 \}$.
\end{proof}

Let $T : V \to W$ be a completely contractive and completely
positive map for matrix ordered operator spaces $V$ and $W$. Then
its unitization $\widetilde{T} : \widetilde{V} \to \widetilde{W}$
defined by
$$\widetilde{T}(x + \lambda 1_{\widetilde{V}}) = T(x) + \lambda
1_{\widetilde{W}}, \qquad x \in V, \lambda \in \bb C$$ is a unital
completely positive map \cite[Lemma 4.9]{W}.

\begin{thm}
Suppose that $\Phi : \cl S \to \cl T$ is a unital completely
positive map for operator systems $\cl S$ and $\cl T$. Then $\Phi
: \cl S \to \cl T$ is a complete order embedding if and only if
the unitization of its dual map $\widetilde{\Phi^*} :
\widetilde{\cl T^*} \to \widetilde{\cl S^*}$ is a complete order
quotient map.
\end{thm}

\begin{proof}
$\Rightarrow )$ Let $f+A \in M_n(\widetilde{\cl S^*})^+$. By Lemma
\ref{unitization}, we have $f_m(x) \ge -I_m \otimes A$ for all $m
\in \mathbb N$ and $x \in M_m(\cl S)_1^+$. We can regard $\cl S$
as an operator subsystem of $\cl T$. By Lemma \ref{extension},
there exists a self-adjoint extension $F : \cl T \to M_n$ such
that $F_m(x) \ge -I_m \otimes A$ for all $m \in \mathbb N$ and $x
\in M_m(\cl T)_1^+$. By Lemma \ref{unitization} again, we have
$$\Phi_n^*(F+A) = f+A \qquad \text{and} \qquad F+A \in
M_n(\widetilde{\cl T^*})^+.$$

$\Leftarrow )$ Suppose that $\Phi_n(x)$ belongs to $M_n(\cl T)^+$
for $x \in M_n(\cl S)$. Let $f : M_n(\cl S) \to \bb C$ be a
positive functional. For any $\varepsilon>0$, there exists a
self-adjoint linear functional $F : M_n(\cl T) \to \bb C$ such
that
$$\Phi^*_n(F)=f \qquad \text{and} \qquad F + \varepsilon I_n \otimes
1_{\widetilde{\cl T^*}} \in M_n(\widetilde{\cl T^*})^+.$$ We have
$$f(x) = F \circ \Phi(x) \ge -\varepsilon n^2 \|x\|.$$ It follows that $x \in
M_n(\cl S)^+$. If $\Phi_n(x) = 0$, then $x \in M_n(\cl S)^+ \cap
-M_n(\cl S)^+ = \{ 0 \}$.
\end{proof}


\end{document}